\documentclass[1 [leqno,11pt]{amsart}
\usepackage{amssymb, amsmath,amsmath,latexsym,amssymb,amsfonts,amsbsy, amsthm}

\numberwithin{equation}{section}

\allowdisplaybreaks

\let\al=\alpha

\let\f=\frac

\let\na=\nabla

\let\pa=\partial

\def\R{\mathbf R}

\def\N{\mathbf N}

\def\cK{\mathcal K}

\def\no{\noindent}

\newcommand{\beq}{\begin{equation}}
\newcommand{\eeq}{\end{equation}}
\newcommand{\ben}{\begin{eqnarray}}
\newcommand{\een}{\end{eqnarray}}
\newcommand{\beno}{\begin{eqnarray*}}
\newcommand{\eeno}{\end{eqnarray*}}

\newtheorem{theorem}{Theorem}[section]

\newtheorem{lemma}[theorem]{Lemma}
\newtheorem{proposition}[theorem]{Proposition}
\newtheorem{corol}[theorem]{Corollary}
\newtheorem{remark}[theorem]{Remark}
\newtheorem{Theorem}{Theorem}[section]

\newtheorem{Lemma}[Theorem]{Lemma}

\newtheorem{Remark}[Theorem]{Remark}

\begin{document}

\title[Separation and local behavior for the Steady Prandtl equation]
{Boundary layer separation and local behavior for the Steady Prandtl equation}

\author{Weiming Shen}
\address{School of Mathematical Sciences, Capital Normal University,100048, Beijing, P. R. China}
\email{wmshen@pku.edu.cn}

\author{Yue Wang}
\address{School of Mathematical Sciences, Peking University, 100871, Beijing, P. R. China}
\email{yuewang37@pku.edu.cn}

\author{Zhifei Zhang}
\address{School of Mathematical Sciences, Peking University, 100871, Beijing, P. R. China}
\email{zfzhang@math.pku.edu.cn}

\begin{abstract}
In the case of favorable pressure gradient, Oleinik proved the global existence of classical solution for the 2-D steady Prandtl equation for a class of positive data. In the case of adverse pressure gradient, an important physical phenomena is the boundary layer separation. In this paper, we prove the boundary layer separation for a large class of Oleinik's data and confirm Goldstein's hypothesis concerning the local behavior of the solution near the separation, which gives a partial answer to open problem 5 proposed by Oleinik and Samokin in \cite{Olei}.
\end{abstract}

\date{\today}

\maketitle

\section{Introduction}
In this paper, we study the 2-D steady Prandtl equation
\begin{equation}\label{eq:Prandtl}
  \left\{
  \begin{aligned}
    &u\pa_x u +v\pa_{y}u-\pa_{y}^2u+\pa_x p=0,\quad x\ge 0,\, y\ge 0,\\
    &\pa_xu+\pa_y v=0,\\
    &u|_{y=0}=v|_{y=0}=0\quad\mbox{and}\quad \displaystyle\lim_{y\to+\infty} u(x,y)=U(x),
    \end{aligned}
  \right.
\end{equation}
where the outer flow $\big(U(x),p(x)\big)$ satisfies
\ben\label{eq:Bernu}
U(x)U'(x)+p'(x)=0.
\een
This system derived by Prandtl is used to describe the behavior of the solution near $y=0$ for the steady Navier-Stokes equations with non slip boundary condition when the viscosity coefficient $\nu$ tends to zero:
\begin{equation}\label{eq:NS}
  \left\{
  \begin{aligned}
    &u^\nu\cdot\na u^\nu-\nu\Delta u^\nu+\na p=f^\nu,\\
    &{\rm div} u^\nu=0,\\
    &u^\nu|_{y=0}=0.
      \end{aligned}
  \right.
\end{equation}
Roughly speaking, away from the boundary, the solution $u^\nu$ can be described by the Euler equations; near the boundary $y=0$,  $u^\nu$ behaves as
\beno
u^\nu(x,y)=\big(u(x,y/\sqrt{\nu}),\sqrt{\nu}v(x,y/\sqrt{\nu})\big),
\eeno
where $(u,v)$ satisfies the Prandtl type equation.

In general case, the inviscid limit problem is still open. However, there are some important progress on the stability  for some special boundary layer flows such as the Blasius flow and shear flow \cite{Guo, GM, GMM}.
For the unsteady Navier-Stokes equations with non slip boundary condition, the inviscid limit was established in the following cases: (1) analytic data \cite{SC1,SC2, WWZ}; (2) the initial vorticity vanishing in a neighborhood of the boundary \cite{Mae, FTZ}; (3) the domain and the initial data having a circular symmetry \cite{LMT, MT}; See the review paper \cite{MM} for a complete introduction.

Although the inviscid limit is not rigorously justified, the Prandtl equation is
a good model in engineering describing the behavior of the solution for the Navier-Stokes equations at large Reynolds number. The goal of this paper is to study the boundary layer separation phenomena based on the steady Prandtl equation \eqref{eq:Prandtl}. See \cite{EE, KVW, WZ} for the singularity formation of the unsteady Prandtl equation.

The existence and regularity of  solution for  the system \eqref{eq:Prandtl}  was proved by Oleinik \cite{Olei}  for a class of positive data $u_0(y)$ prescribed at $x=0$. Indeed, if $u_0(y)>0$, the system \eqref{eq:Prandtl} could be viewed as a parabolic equation with the initial boundary condition where the variable $x$ is the time direction. For the favourable pressure gradient $p'(x)\le 0$, the solution is global in $x$(see \cite{XZ} for the unsteady case). On the other hand, for the adverse pressure gradient $p'(x)>0$, the boundary layer separation could occur in a finite time. Goldstein made a formal asymptotic analysis for the solution near
the separation point, i.e., $\pa_yu(x^*,0)=0$ based on
{\bf three key assumptions}(see Page 47 in \cite{G}):{\it  (1) there is a singularity at separation;
(2) there is a finite value of $u$ at separation for $y\neq 0$;
(3) $a_2=\f12$(i.e., $u=a_2y^2+a_3y^3+\cdots$ at $x=0$). Related to assumption (3), Professor Hartree found (empirically) that in his solution $\pa_yu(x,0)$ behaved near $x = 0$(corresponding to separation) like a
multiple of $x^r$, where $r$ is certainly less than 1 and greater than $\f14$. }Thus, he made the following formal expansion
\beno
&&u(0,y)=\f12 y^2+a_3y^3+\cdots,\\
&&\pa_yu(x,0)=\al_1x^\f12+\al_2x^\f34+\al_3 x+\al_4x^\f54+\cdots.
\eeno

In a review paper \cite{E}, E claimed an important progress in an unpublished paper(joint with Cafferalli): if the initial data $u_0(y)$ and the pressure $p(x)$ satisfy
\beno
u_0(y)^2-\f32 \pa_yu_0(y)\int_0^yu_0(z)dz\ge 0,\quad p'(x)\ge \al>0,
\eeno
then there exists an $x^*>0$ so that the solution can not be extended to $x>x^*$; moreover, the sequence of $u_\lambda$ defined by
\beno
u_\lambda(x,y)=\lambda^{-\f12}u(x^*-\lambda x, \lambda^\f14y)
\eeno
is compact in $C^0(\R^+\times \R^+)$. In some sense, this means that
the solution behaves as
\beno
u(x,y)\sim (x^*-x)^\f12U_0\Big(\f y {(x^*-x)^\f14}\Big),\quad x<x^*.
\eeno

In a recent important work \cite{DM}, Dalibard and Masmoudi prove the boundary layer separation for a class of special data and $p'(x)=1$, and show that the solution behaves near the separation:
\beno
\pa_yu(x,0)\sim (x^*-x)^\f12,\quad x<x^*.
\eeno
This result is compatible with Goldstein's assumption (3).

In the case of adverse pressure gradient, for general Oleinik's data ensuring the existence of the solution, whether the boundary layer separation can occur and the local behavior of the solution near the separation is a long-standing problem. The following open problem was proposed by  Oleinik and Samokin 
(P.501 in \cite{Olei}):\smallskip

{\it 
It would be interesting to study the local structure of the solution 
of the Prandtl system in the vicinity of the separation point.
}

\smallskip

In this paper, we prove that  the boundary layer separation can occur in a finite time for a large class of Oleinik's data in the case of the adverse pressure gradient. Moreover, we study the local behavior of the solution near the separation and confirm Goldstein's assumption (3): for $x$ close to $x^*$,
\beno
\pa_yu(x,0)\le C(x^*-x)^\f14.
\eeno
Furthermore, there exists $y(x)$ satifying $\int_0^{y(x)}u(x,y')dy'\le C(x^*-x)^\f34$ for any $x$ close to $x^*$  so that
\beno
\pa_yu\big(x,y(x)\big)\sim (x^*-x)^\f14.
\eeno
This result together with Dalibard and Masmoudi's result shows that
the solution has a different separation rate when the point approaches the separation point along a different curve. This complex local behavior of the solution near the separation is perhaps related to Stewartson's observation that some coefficients in the asymptotic expansion are not uniquely determined \cite{Ste}. \smallskip

The complete results and their proof will be presented in subsequent sections. In section 2, we review some classical results on the existence and regularity of the solution. In section 3, we prove the boundary layer separation. In section 4, we prove a lower bound on the separation rate.
In section 5, we study the local behavior of the solution near the separation. In section 6, we extend our result to general adverse pressure gradient.

\section{Oleinik's result and Von  Mises transformation}

Let us first introduce a class of data denote by $\cK$, which satisfies
\begin{align*}
 &u\in C_b^{3,\alpha}\big([0,+\infty)\big)(\al>0),\quad u(0)=0,\,\,u_y(0)>0,\,\,u_y(y)\geq0\, \text{for}\, y\in[0,+\infty),\\
 &\displaystyle\lim_{y\to+\infty} u(0,y)=U(0)>0,\quad u_{yy}(y)-\partial_xp(0)=O(y^2).
\end{align*}
For the data $u_0\in \cK$, Oleinik proved the existence of solution of the system \eqref{eq:Prandtl}(see Proposition 1.1 in \cite{DM} and Theorem 2.1.1 in \cite{Olei}).

\begin{proposition}\label{prop:olei} If
$u_0\in \cK,$ then there exists $X>0$ such that the steady Prandtl equation \eqref{eq:Prandtl} admits a solution $u\in C^1([0,X)\times\R_+)$ with the following properties:

\begin{itemize}
\item[1.] Regularity: $u$ is bounded and continuous in $[0,X]\times[0,+\infty); \,u_{y},u_{yy}$ are bounded and continuous in $[0,X)\times\R_+$; and $v,v_{y},u_{x}$ are locally bounded and continuous in $[0,X)\times\R_+.$
\item[2.] Non-degeneracy: $u(x,y)>0$ in $[0,X)\times (0,+\infty)$ and for all $\bar{x}<X,$ there exists $y_0>0,m>0$ so that
\beno
\partial_yu(x,y)\geq m\quad \text{in}\,\, [0,\bar{x}]\times[0,y_0].
\eeno
\item[3.] Global existence: if $p'(x)\le 0$, then the solution is global in $x$.\end{itemize}
\end{proposition}

Oleinik's proof is based on the Von Mises transformation:
\begin{align}\label{eq:VM}
 \psi(x,y)=\int_0^yu(x,z)dz,\quad w=u^2.
\end{align}
A direct calculation shows that
  \begin{align}\label{PvM}
  \begin{split}
& 2\partial_y u =\partial_\psi w,\quad2\partial^2_{y} u=\sqrt{w}\partial^2_{\psi}w.
 \end{split}
\end{align}
Then the new unknown $w(x,\psi)$ satisfies
  \begin{align}\label{eq:Prandtl-VM}
 \pa_xw-\sqrt{w}\partial^2_{\psi}w=-2\pa_xp\quad\text{in}\quad[0,X)\times \R_+,
 \end{align}
together with
\ben\label{eq:Prandtl-bc}
w(0,\psi)=w_0(\psi)=u_0(y)^2,\quad w(x,0)=0,\quad \displaystyle\lim_{\psi\to+\infty}w(x,\psi)=U(x)^2.
\een

From Lemma 2.1.9 and Lemma 2.1.11 in \cite{Olei}(or Lemma 3.1 in \cite{DM}), we know that

\begin{lemma}\label{lem:w-bound}
Let $u$ be a solution constructed in Proposition \ref{prop:olei}. It holds that
\begin{itemize}

\item[1.] the solution $w(x,\psi)$ is increasing with respect to $\psi.$

\item[2.] for any $x\in [0,X_1], X_1<X$, there exists $C_{X_1}>0$ so that
\begin{align*}
&|\partial_\psi w(x,\psi)|\le C_{X_1}\quad \psi\ge 0,\\
&|\partial^2_\psi w(x,\psi)|+|\partial^3_\psi w(x,\psi)|\leq C_{X_1}\quad \psi\ge 1.
\end{align*}

\item[3.] for any $x\in[0,X)$,
 \begin{align*} \displaystyle\lim_{\psi\to+\infty} \partial_\psi w(x,\psi)=0,\quad \displaystyle\lim_{\psi\to+\infty} \partial^2_\psi w(x,\psi)= 0.
\end{align*}
\end{itemize}
\end{lemma}

 The following lemma comes from  Lemma 4 in \cite{MS}, which is essentially given by Oleinik \cite{Olei}.

\begin{lemma}\label{lem:extension}
Let  $k_1=\min_{0\leq x\leq x_1} U(x)^2$ and $k_2=\max_{0\leq x\leq x_1}  |p'(x)|$ for some $x_1>0$. Assume that  for some $k_0,k>0$,
\begin{align}\label{ex1dr}
 \inf\{\partial_\psi w_0(\psi):0\leq \psi\leq k_0 \}\geq k.
\end{align}
Then there exists a positive constant $X_0$ depending only on $k_0,k,k_1$ and $k_2$ so that the local solution $w(x,\psi)$ to \eqref{eq:Prandtl-VM} exists in $\{(x,\psi)|(x,\psi)\in[0,X_0]\times[0,+\infty)\}$.
\end{lemma}

 In the sequel, we first consider the case of $\pa_xp=1$ so that
 \beno
 U(x)=\sqrt{2(x_0-x)}\quad \text{for some}\,\,x_0>0.
 \eeno
  We denote by $X^*$ the maximal existence time of the solution in Proposition \ref{prop:olei} and $X_*=\min(X^*, x_0)$. We say that $x_s$
 is a separation point of $u$ if $\pa_yu(x,0)\to 0$  as $x\to x_s$.
 \smallskip

 In section 6, we will extend our result to general adverse pressure $\pa_xp\ge c>0$.

\section{Boundary layer separation}\label{LS}

In this section, we prove the boundary layer separation for a large class of  data in $\cK$.

\begin{theorem}\label{thm:SP}
Fix any $\mu\in(0,1).$ Let $u$ be a solution constructed in Proposition \ref{prop:olei} with $u_0\in \cK$ satisfying
\begin{align}\label{spcdt}
\|\partial_y u_0\|_{L^\infty(0,y_0)}\leq \frac{1}{2}\epsilon_0x_0^{\frac{1}{4}},
\end{align}
where $y_0$ is determined by
$$Bx_0^{\frac{3}{4}}=\psi_0=\int_0^{y_0}u_0(z)dz, $$
and $\epsilon_0,B$ are positive constants depending only on $\mu.$
Then there exists a separation point $x_s=X^*$ with $X^*<\frac{\mu}{2}x_0$.
\end{theorem}

\begin{Remark}\label{spcrk}
For a large class of data in $\cK$, the condition \eqref{spcdt} is satisfied. Indeed, $u_0\in \cK$ implies that $\|\partial_y u_0\|_{L^\infty(0,+\infty)}<+\infty$ and $\sqrt{2x_0}=\displaystyle\lim_{y\to+\infty} u_0(y)$. Thus, given $u_0\in \cK$, we can take $x_0$ large enough compared with $\|\partial_y u_0\|_{L^\infty(0,+\infty)}$ but without requirement on the size of $\|u_0\|_{L^\infty(0,+\infty)}$ so that
\eqref{spcdt} holds.  On the other hand,  given $x_0>0$, one can find
$u_0\in \cK$ with small slope so that \eqref{spcdt}  is satisfied.
\end{Remark}

\subsection{One-side estimate on $\pa_y^2u$}

The following lemma plays an important role in this paper.

 \begin{Lemma}\label{lem:one side}
 If $\partial^2_{y}u_0\geq -C_1$ and $X^*< x_0,$ there exists a positive constant $C_2$ depending on $C_1$  so that
\begin{align*}
  \partial^2_{y}u(x,y)\geq -C_2\quad \text{in} \quad\big[0,X^*\big)\times [0,+\infty).
\end{align*}
\end{Lemma}

\begin{proof}
By \eqref{PvM}, it suffices to show that there exists a positive constant $M$ so that
$$
\sqrt{w}\partial^2_{\psi}w\geq -M\quad \text{in} \quad[0,X^*)\times [0,+\infty).
$$
For any fixed $\bar{x}\in (0,X^*)$, we denote
$$
D_{\bar{x}}=\big\{(x,\psi)|(x,\psi)\in[0,\bar{x})\times [0,+\infty)\big\}.
$$

Now we consider the function $f=\partial_x w=\sqrt{w}\partial^2_{\psi}w-2$ in $D_{\bar x}$. Thanks to $\sqrt{w}\partial^2_{\psi}w=2\pa_y^2u$ and $\partial^2_{y}u_0\geq -C_1$, we have
\beno
f(0,\psi)=\sqrt{w}\partial^2_{\psi}w|_{x=0}-2\geq -2C_1-2.
\eeno
Due to $\pa_y^2u|_{y=0}=2$, we have
\beno
f(x,0)=0\quad \text{for}\quad x\in [0,X^*),
\eeno
and by Lemma \ref{lem:w-bound},
\begin{align*}
\displaystyle\lim_{\psi\to+\infty} f= \displaystyle\lim_{\psi\to+\infty}\sqrt{w}\partial^2_{\psi}w-2=-2\quad \text{for} \quad x\in [0,X^*).
\end{align*}
Therefore, we may assume that $f$ achieves its negative minimum on $D_{\bar x}$ at a point $z_{min} \in (0,\bar x]\times (0,+\infty)$. Then $f(z_{min})<0$ and at $(x,\psi)=z_{min}$,
\begin{align}\label{ppm}
   \partial_xf\leq 0,\quad \partial_{\psi}f=0,\quad \partial^2_{\psi}f\geq 0.
\end{align}
Taking $\pa_x$ to the equation \eqref{eq:Prandtl-VM}, we find that
\begin{align*}
\partial_xf=\sqrt{w}\partial^2_{\psi}f+\frac{f}{w}+\frac{f^2}{2w} \quad \text{in}\quad (0,X^*)\times (0,+\infty),
\end{align*}
from which and \eqref{ppm}, we infer that
\begin{align*}
\frac{f}{w}+\frac{f^2}{2w} \leq 0\quad \text{at}\quad (x,\psi)=z_{min},
\end{align*}
which gives
$$-2\leq f(z_{min})<0.$$
Since $\bar{x}$ is chosen arbitrarily, we conclude our conclusion.
\end{proof}

A direct consequence of this lemma is the following corollary.
\begin{corol}\label{cor:SP}
If $X^*<x_0,$ then it holds that
\begin{align}\label{abc}
\partial_y u(x,0)\rightarrow0\quad\text{as}\quad x\rightarrow X^*.
\end{align}
\end{corol}

\begin{proof}
Since  $u_0$ is increasing in $y$, we have $\partial_yu(x,0)\geq 0$
for $ x\in(0,X^*).$ If \eqref{abc} does not hold, then there exists a positive constant $\epsilon_0$ so that for any $n\in \N_+,$ there exists a $x_n\in (X^*-\f1n, X^*)$ with
\beno
\partial_y u(x_n,0)\geq \epsilon_0,
\eeno
which along with Lemma \ref{lem:one side} implies that
\beno
\inf\Big\{\partial_yu(x_n,y):0\leq y\leq \frac{\epsilon_0}{2C_2}\Big\}\geq \frac{\epsilon_0}{2}.
\eeno
On the other hand, $U(x_n)^2=2(x_0-x_n)\geq 2(x_0-X^*)>0$. For any $n,$ we can take $x_n$ as an initial time. That is, by Lemma \ref{lem:extension}, there exists $\delta>0$ depending only on $\epsilon_0$ and $x_0-X^*>0$ so that $u$ can be extended to $[0,x_n+\delta]$. However, $|X^*-x_n|\rightarrow 0$ as $n\rightarrow \infty$ so that the solution can be extended after $x=X^*$ by taking $n$ big enough, which is a contradiction. Then the corollary follows.
\end{proof}

\subsection{Proof of Theorem \ref{thm:SP}}

Now we are in a position to prove Theorem \ref{thm:SP}.

\begin{proof}[Proof of Theorem \ref{spcdt}]
Take a smooth cut-off function $\varphi(\psi)$ so that
\beno
\varphi\equiv 1\quad \text{in} \,\,\big[0,\frac{\delta}{2}\big],\quad\varphi\equiv 0\quad \text{in} \,\,(\delta,+\infty),\quad0\leq \varphi\leq 1 \quad \text{in}\,\,[0,+\infty).
\eeno
Therefore,
\begin{align}\label{phi2dr}
   |\partial^2_{\psi}\varphi|\leq \frac{C}{\delta^2},
\end{align}
where we take $\delta=x_0^{\frac{3}{4}}B$ with $B$ large to be determined.

First of all, we have $\partial_\psi w\geq 0$ and
\begin{align}\label{upbdw}
   0\leq (w(x,\psi))^{\frac{3}{2}}\leq (2(x_0-x))^{\frac{3}{2}}\leq Cx_0^{\frac{3}{2}},\quad \text{in} \,\,[0,X_*)\times [0,+\infty).
\end{align}
We denote
$$
D_\mu=\Big\{(x,\psi):(x,\psi)\in\big[0,(\mu+(1-\mu)\frac{99}{100})X_*\big]\times[0,+\infty)\Big\}.
$$
By \eqref{eq:Prandtl-VM}, we have
$$
\partial_x\int_0^{\delta}w\varphi d\psi-\int_0^{\delta}\sqrt{w}\partial^2_{\psi}w\varphi d\psi=-2\int_0^{\delta}\varphi d\psi.
$$
Thanks to $\sqrt{w}\partial_{\psi}w(x,0)=0$, we get by integration by parts that
$$\partial_x\int_0^{\delta}w\varphi d\psi+\int_0^{\delta}\sqrt{w}\partial_{\psi}(\sqrt{w})^2\partial_{\psi}\varphi d\psi+\int_0^{\delta}2\sqrt{w}(\partial_{\psi}\sqrt{w})^2\varphi d\psi=-2\int_0^{\delta}\varphi d\psi.$$
Using the facts that
\begin{align*}
  &\int_0^{\delta}2\sqrt{w}(\partial_{\psi}\sqrt{w})^2\varphi d\psi\geq 0,\\
  &\int_0^{\delta}\sqrt{w}\partial_{\psi}(\sqrt{w})^2\partial_{\psi}\varphi d\psi=-\frac{2}{3}\int_0^{\delta}w^{\frac{3}{2}}\partial^2_{\psi}\varphi d\psi,
\end{align*}
we infer from  \eqref{phi2dr} and \eqref{upbdw}  that
$$
0\leq \int_0^{\delta}w\varphi d\psi\leq \int_0^{\delta}w_0\varphi d\psi-2x\int_0^{\delta}\varphi d\psi+\frac{2}{3}\int_0^{\delta}d\psi x_0^{\frac{3}{2}}\frac{C}{\delta^2}x_0.
$$
This shows that
\begin{align*}
\delta x\leq 2x\int_0^{\delta}\varphi d\psi \leq \frac{\delta^2}{2}\|\partial_\psi w_0 \|_{L^\infty(0,\delta)}+C\frac{x_0^{\frac{5}{2}}}{\delta},
\end{align*}
where we used $w_0(0)=0$ and $2\partial_y u_0=\partial_\psi w_0$. Hence,
\beno
x\leq \frac{\delta}{2}\|\partial_\psi w_0 \|_{L^\infty(0,\delta)}+C\frac{x_0^{\frac{5}{2}}}{\delta^2}.
\eeno
By \eqref{spcdt} and $2\partial_y u_0=\partial_\psi w_0 $, we have \begin{align*}
\|\partial_\psi w_0\|_{L^\infty(0,\psi_0)}\leq \epsilon_0x_0^{\frac{1}{4}},
\end{align*}
where $\psi_0=Bx_0^{\frac{3}{4}}=\delta.$
Therefore, it holds that
\ben\label{eq:x-est1}
x\leq \frac{Bx_0^{\frac{3}{4}}}{2}\epsilon_0x_0^{\frac{1}{4}}+C \frac{x_0}{B^2}=\frac{B}{2}\epsilon_0x_0+C\frac{x_0}{B^2}.
\een
Now we first take $B$ large enough so that
$$ \frac{C}{B^2}\leq \frac{\mu^2}{4},$$
and then take $\epsilon_0$ small enough so that
$$\frac{B}{2}\epsilon_0\leq \frac{\mu^2}{4}.
$$
Then we deduce from \eqref{eq:x-est1} that $x\leq  \frac{\mu^2}{2}x_0$,
which implies
\beno
\mu X_*<\Big(\mu+(1-\mu)\frac{99}{100}\Big)X_*\le  \frac{\mu^2}{2}x_0.
\eeno
That is, $X_*\leq  \frac{\mu}{2}x_0.$ By the definition of $X_*,$ we have
$$\min\{X^*,x_0\}=X_*\leq\frac{\mu}{2}x_0.$$
Therefore, $X^*\leq\frac{\mu}{2}x_0.$ Then the theorem follows from
Corollary \ref{cor:SP}.
\end{proof}

\section{Goldstein's hypothesis on the separation rate}\label{AB}

In this section, we confirm Goldstein's hypothesis concerning the separation rate of boundary layer.

\begin{theorem}\label{thm:lower}
Let $u$ be a solution constructed in Proposition \ref{prop:olei} with
$X^*<x_0$. Then there exists a positive constant $C$ so that
\begin{align}\label{1/4}
\partial_yu(x,0)\leq C\big(X^*-x\big)^{\frac{1}{4}} \quad \text{for}\quad x\in \big(x_{near},X^*\big),
\end{align}
for some $x_{near}$ close enough to $X^*$ determined in Lemma \ref{e1/4}.
\end{theorem}

The idea is that we find a curve $(x,y(x))$ such that $\partial_yu(x,y(x)) \leq C\big(X^*-x\big)^{\frac{1}{4}}$, then
using one-side estimate in Lemma \ref{lem:one side}, we can deduce the  separation rate of $\partial_yu(x,0).$

\begin{lemma}\label{e1/4}If $X^*<x_0,$ then there exists a positive function $\mu(x)$ defined on $\big(x_{near},X^*\big)$ so that for all $ x\in\big(x_{near},X^*\big),$
\begin{align}
    \begin{split}
       0 &< \mu(x)< C_3\big(X^*-x\big)^{\frac{1}{4}},\\
        0& <\displaystyle\min_{0\leq y\leq \mu(x)^{\frac{1}{4}}}\partial_yu(x,y) \leq C_3\big(X^*-x\big)^{\frac{1}{4}},
     \end{split}
\end{align}
where $C_3$ is a positive constant and $x_{near}$ is any fixed point close to $x^*$ so that $\partial_yu(x,0)\leq \f12$ for  $x\in\big(x_{near},X^*\big).$
 \end{lemma}

 \begin{Remark}
 Due to Corollary \ref{cor:SP}, $x_{near}$ is well defined.
 \end{Remark}

\begin{proof}
For any $x\in \big(x_{near},X^*\big)$, we define
$$
I_x=\big\{v|v\leq\displaystyle\min_{0\leq y\leq v^{\frac{1}{4}}}(\partial_yu)^4(x,y)\big\}.
$$
Then $0\in I_x.$ Thanks to the choice of $x_{near}$, we have
\begin{align}\label{mu<1}
 v\leq\displaystyle\min_{0\leq y\leq v^{\frac{1}{4}}}(\partial_yu)^4(x,y)\leq (\partial_yu)^4(x,0)\leq \frac{1}{2^4}<1.
\end{align}
Therefore, $I_x$ is  nonempty and bounded. Thus, we can take
$$\mu(x)=\sup_{y\in I_x} y,$$
which is bounded and positive, since $0<(\partial_yu(x,0))^4$ for $x<X^*$.
Moreover, we have
 \begin{align}\label{=}
 \mu(x)=\displaystyle\min_{0\leq y\leq (\mu(x))^{\frac{1}{4}}}(\partial_yu)^4(x,y),
\end{align}
which implies that $\displaystyle\min_{0\leq y\leq (\mu(x))^{\frac{1}{4}}}(\partial_yu)^4(x,y)>0.$ Indeed, \eqref{=} follows from the fact that if $g(y)$ is continuous in $y,$ then $\displaystyle\min_{0\leq y\leq \alpha}g(y)$ is continuous with respect to $\alpha.$

For any $x_1\in\big(x_{near},X^*\big),$ we introduce
$$
\widetilde{u}(\tilde{x},y)=u(x_1+\tilde{x},y),\quad \tilde{x}\in\big[0,X^*-x_1\big).
$$
Then $\widetilde{u}$ is a solution to \eqref{eq:Prandtl} with
\beno
\widetilde{u}_0(y)=u(x_1,y)\quad \text{and} \quad\widetilde{U}(\tilde{x})=\sqrt{2(x_0-x_1-\tilde{x})}.
\eeno
Let $\mu=\mu(x_1)$ for convenience and denote
\beno
\widetilde{u}_\mu(\tilde{x},y)=\frac{1}{\sqrt{\mu}}\widetilde{u}(\mu\tilde{x},\mu^{\frac{1}{4}}y),\quad \tilde{x}\in\Big[0,\frac{X^*-x_1}{\mu}\Big).
\eeno
Then $\widetilde{u}_\mu$ is a solution of \eqref{eq:Prandtl} with \begin{align}\label{mubdy}
(\widetilde{u}_\mu)_0(y)=\frac{1}{\sqrt{\mu}}u(x_1,\mu^{\frac{1}{4}}y)\quad \text{and}\quad \widetilde{U}_\mu(\tilde{x})=\sqrt{2\Big(\frac{x_0-x_1}{\mu}-\tilde{x}\Big)}.
\end{align}
We denote by $A_\mu$ the maximal existence time of $\widetilde{u}_\mu$. Then  we have
\begin{align}\label{upb}
    \mu A_\mu+x_1\leq X^*.
\end{align}
On the other hand, by \eqref{=} and \eqref{mubdy}, we have
\beno
\partial_y(\widetilde{u}_\mu)_0(y)=\frac{1}{\mu^{\frac{1}{4}}}\pa_y u(x_1,\mu^{\frac{1}{4}}y)=\frac{\pa_y u(x_1,\mu^{\frac{1}{4}}y)}{\displaystyle\min_{0\leq y\leq \mu^{\frac{1}{4}}}\partial_y u(x_1,y)}\geq 1\quad y\in[0,1],
\eeno
and
\beno
\widetilde{U}_\mu(0)\geq\sqrt{2\Big(\frac{x_0-x_1}{\mu}\Big)}>\sqrt{2(x_0-X^*)},
\eeno
where we used $\mu\le 1$ due to \eqref{mu<1}.
Then  Lemma \ref{lem:extension}  ensures that there exists a positive constant $B$ depending only on $x_0-X^*$ so that $(\widetilde{u}_\mu)(\tilde{x},y)$ can be extended after $\tilde{x}=B$. Then $B\leq A_\mu$ and hence by \eqref{upb}, we have
\beno
\mu(x_1)=\mu\leq\frac{1}{B}\big(X^*-x_1\big).
\eeno

Finally, since $x_1$ is chosen arbitrarily and $B$ depends only on $x_0-X^*$, we have $\mu(x)\leq\frac{1}{B}\big(X^*-x\big).$ Then by \eqref{=}, we get
\beno
\displaystyle\min_{0\leq y\leq(\mu(x))^{\frac{1}{4}}}\partial_y u(x,y)=(\mu(x))^{\frac{1}{4}} \leq C_3\big(X^*-x\big)^{\frac{1}{4}}.
\eeno
This completes the proof of the lemma.
\end{proof}

Now we prove Theorem \ref{thm:lower}.

\begin{proof}
For any $x\in\big(x_{near},X^*\big)$, we define
$$
T_x=\big\{v|\displaystyle\min_{0\leq y\leq v}\partial_yu(x,y)=\frac{1}{2}\partial_yu(x,0)\big\}.
$$
Using the fact that for $x\in\big(x_{near},X^*\big),$
\beno
\displaystyle\lim_{ v\rightarrow +\infty }\displaystyle\min_{0\leq y\leq v}\partial_yu(x,y)=0<\frac{1}{2}\partial_yu(x,0),
\eeno
and $\displaystyle\min_{0\leq y\leq v}\partial_yu(x,y)$ is continuous with respect to $v$, we deduce that the set $T_x$ is nonempty and bounded.
We set $v(x)=\sup_{y\in T_x}y>0$.

First of all, if $v(x)\geq (\mu(x))^{\frac{1}{4}}, $ then we get by Lemma \ref{e1/4} that
$$\partial_yu(x,0)=2\displaystyle\min_{0\leq y\leq v(x)}\partial_yu(x,y)\leq2\displaystyle\min_{0\leq y\leq (\mu(x))^\frac{1}{4}}\partial_yu(x,y)\leq 2C_3\big(X^*-x\big)^\frac{1}{4}.$$
While, if  $v(x)< (\mu(x))^{\frac{1}{4}}<C_3\big(X^*-x\big)^\frac{1}{4}$, then we have
$$\partial_yu(x,0)\leq 3C_2C_3\big(X^*-x\big)^\frac{1}{4}.$$
Indeed, if $\partial_yu(x,0)> 3C_2C_3\big(X^*-x\big)^\frac{1}{4}$,
then there exists $\xi_x\in(0,v(x))\subset\big(0,C_3\big(X^*-x\big)^\frac{1}{4}\big)$ so that
\beno
\partial_{y}^2u(x,\xi_x)=\frac {-\f12\pa_yu(x,0)} {v(x)}\leq-\frac{ \frac{1}{2}\partial_yu(x,0)}{C_3(X^*-x)^\frac{1}{4}}\leq -\frac{3}{2}C_2,
\eeno
which contradicts with Lemma \ref{lem:one side}.
\end{proof}

\section{Local behavior near the separation}
\label{ABl}

In this section, we further study the local behavior of the solution near the separation. The following theorem gives a partial answer to open problem 5 proposed by Oleinik and Samokin (P.501, \cite{Olei}):\smallskip

{\it 
\no It would be interesting to study the local structure of the solution 
of the Prandtl system in the vicinity of the separation point.
}

 \begin{theorem}\label{thm:behavior}
 Let $u$ be a solution constructed in Proposition \ref{prop:olei}. If $u$ satisfies $\partial_{y}^2u\leq M_0$ in $[0,X^*)\times \mathbf{R}_+$, then for any $\bar{x}<X^*<x_0$,
there exist a point $(\tilde{x},\psi_{\tilde{x}})\in [\bar{x},X^*)\times\big[0,(X^*-\tilde{x})^{\frac{3}{4}}\big)$ and two positive constants $c,C$ independent of the choice of $\bar{x}$ so that
\begin{align}\label{km24.1}
C(X^*-\tilde{x})^{\frac{1}{4}}
\geq\partial_\psi w(\tilde{x},\psi_{\tilde{x}})\geq c(X^*-\tilde{x})^{\frac{1}{4}}.
\end{align}\end{theorem}

In Proposition \ref{prop:uyy-upper}, we will provide a sufficient condition ensuring that  $\partial_{y}^2u\leq M_0$ in $[0,X^*)\times \mathbf{R}_+$ holds. However, we don't need the assumption $\pa_y^2u\le M_0$ in the following lemma.

\begin{lemma}\label{lem:lower}
Let
\begin{align}\label{km1}
   k\geq \frac{1}{m+1},\quad m\geq \frac{1}{3},\quad \delta(x)=(X^*-x)^k.\end{align}
Then it holds that for any $\bar{x}<X^*< x_0$,
there exists a point $(\tilde{x},\psi_{\tilde{x}})\in[\bar{x},X^*)\times[0,\delta(x))$ and a positive constant $c$ independent of the choice of $\bar{x}$ such that
\begin{align}\label{km2}
\partial_\psi w(\tilde{x},\psi_{\tilde{x}})\geq c(X^*-\tilde{x})^{km}.
\end{align}
\end{lemma}
\begin{remark}
As $k\geq \frac{1}{m+1}$ and $m\geq \frac{1}{3},$ we have
\begin{align}
  km\geq  \frac{m}{m+1}\geq \frac{1}{4}.
\end{align}
\end{remark}
\begin{proof}
Take any $\bar{x}\in [X^*-1,X^*)$ and choose $\varphi$ to be a smooth non-increasing cut-off function so that
$$
\varphi\equiv 1\quad \text{in}\,\,[0,\frac{1}{2}],\quad\varphi\equiv 0\quad \text{in}\,\,(1,+\infty),\quad0\leq \varphi\leq 1 \quad \text{in}\,\,[0,+\infty).
$$
Let $\bar{\varphi}(x,\psi)=\varphi\Big(\frac{\psi}{\delta(x)}\Big)$.
Then we have
\begin{align}\label{phibar2dr}
 &|\partial_{\psi}  \bar{\varphi}|\leq \frac{C}{\delta(x)}, \quad|\partial^2_{\psi}  \bar{\varphi}|\leq \frac{C}{\delta(x)^2},\\
&\label{phibarsign}
    -\partial_x \bar{\varphi}\geq 0,\quad -\partial_x \delta(x)\geq0.
\end{align}

It follows from \eqref{eq:Prandtl-VM} that
\begin{align}\label{barMSP}
\int_0^{\delta(x)}\partial_xw\bar{\varphi} d\psi-\int_0^{\delta(x)}\sqrt{w}\partial^2_{\psi}w\bar{\varphi} d\psi=-2\int_0^{\delta(x)}\bar{\varphi} d\psi.
\end{align}

First of all, we get by \eqref{phibarsign} that
 \begin{align*}
   \int_0^{\delta(x)}\partial_xw\bar{\varphi} d\psi&=\partial_x\Big(\int_0^{\delta(x)}w\bar{\varphi} d\psi\Big)-w\bar{\varphi}\partial_x \delta -\int_0^{\delta(x)}w\partial_x \bar{\varphi}d\psi\\
   &\geq\partial_x\Big(\int_0^{\delta(x)}w\bar{\varphi} d\psi\Big),
 \end{align*}
 which gives
 \begin{align}\label{spartx}
  \int_{\bar{x}}^{X^*} \int_0^{\delta(x)}\partial_xw\bar{\varphi} d\psi dx\geq- \int_0^{\delta(\bar{x})}w\bar{\varphi}(\bar{x},\psi) d\psi.
 \end{align}
Secondly, we get by integration by parts that
\begin{align*}
-\int_0^{\delta(x)}\sqrt{w}\partial^2_{\psi}w\bar{\varphi} d\psi&=\int_0^{\delta(x)}2\sqrt{w}(\partial_{\psi}\sqrt{w})^2\bar{\varphi} d\psi
+\int_0^{\delta(x)}\sqrt{w}\partial_{\psi}(\sqrt{w})^2\partial_{\psi}\bar{\varphi} d\psi
\\&\geq -\frac{2}{3}\int_0^{\delta(x)}w^{\frac{3}{2}}\partial^2_{\psi}\bar{\varphi} d\psi,
\end{align*}
which gives
\begin{align}\label{partpsipsi}
-\int_{\bar{x}}^{X^*}\int_0^{\delta(x)}\sqrt{w}\partial^2_{\psi}w\bar{\varphi} d\psi dx\geq -\frac{2}{3}\int_{\bar{x}}^{X^*}\int_0^{\delta(x)}w^{\frac{3}{2}}\partial^2_{\psi}\bar{\varphi} d\psi dx.
\end{align}
Thirdly, we have
\begin{align}\label{bardelta}
 -2\int_{\bar{x}}^{X^*}\int_0^{\delta(x)}\bar{\varphi} d\psi dx\leq -\int_{\bar{x}}^{X^*}\delta(x)dx.
\end{align}
Putting \eqref{barMSP}, \eqref{spartx}, \eqref{partpsipsi} and \eqref{bardelta} together, we infer that
\begin{align}\label{barcond}
     \int_0^{\delta(\bar{x})}w\bar{\varphi}(\bar{x},\psi) d\psi +\frac{2}{3}\int_{\bar{x}}^{X^*}\int_0^{\delta(x)}w^{\frac{3}{2}}\partial^2_{\psi}\bar{\varphi} d\psi dx\geq\int_{\bar{x}}^{X^*}\delta(x)dx.
\end{align}

Next we argue by contradiction. Assume that
\begin{align}\label{barcondas}
 |\partial_\psi w(x,\psi)|\leq \epsilon_{k,m} \delta(x)^m,\quad x\in [\bar{x},X^*)\times[0,\delta(x)),
\end{align} \
where $\epsilon_{k,m}$ is a small positive constant to be determined.  Then by \eqref{barcond} and \eqref{phibar2dr}, we have
\begin{align*}
    \frac{1}{k+1}(X^*-\bar{x})^{k+1}&\leq \epsilon_{k,m}\frac{\delta(\bar{x})^{m+2}}{2}
   + \frac{2}{3}\int_{\bar{x}}^{X^*}\int_0^{\delta(x)}\epsilon_{k,m}^{\frac{3}{2}}\delta(x)^{\frac{3}{2}m}\psi^{\frac{3}{2}}\frac{C}{\delta(x)^{2}} d\psi dx\\&\leq \frac{1}{100} \frac{1}{k+1}(X^*-\bar{x})^{k(m+2)}+\frac{1}{100} \frac{1}{k+1}(X^*-\bar{x})^{k(\frac{3m}{2}+\frac{1}{2})+1},
\end{align*}
by taking $\epsilon_{k,m}$ small depending only on $k,m$.
Since $k\geq \frac{1}{m+1},m\geq \frac{1}{3},$ we have
\begin{align}
  k\big(\frac{3m}{2}+\frac{1}{2}\big)+1\geq k+1,\quad k(m+2)\geq k+1.
\end{align}
This shows that
\begin{align*}
   (X^*-\bar{x})^{k+1}\leq \frac{2}{100} (X^*-\bar{x})^{k+1},
\end{align*}
which leads to a contradiction.
\end{proof}

Now we prove Theorem \ref{thm:behavior}.

\begin{proof}
 Take $m=\frac{1}{3},k=\frac{1}{\frac{1}{3}+1}=\frac{3}{4}$ in Lemma \ref{lem:lower}. Then for any $\bar{x}<X^*<x_0$,
there exist a point $(\tilde{x},\psi_{\tilde{x}})\in[\bar{x},X^*)\times[0,(X^*-\tilde{x})^{\frac{3}{4}})$ and a positive constant $c_0$ independent of the choice of $\bar{x}$ such that
\begin{align}
\partial_\psi w(\tilde{x},\psi_{\tilde{x}})\geq c_0(X^*-\tilde{x})^{\frac{1}{4}}.
\end{align}

\no$\mathbf{Case\,1}$. If $\psi_{\tilde{x}}=0,$ then by \eqref{PvM} and Theorem \ref{thm:lower}, we have
\begin{align*}
\partial_\psi w(\tilde{x},\psi_{\tilde{x}})\leq C(X^*-\tilde{x})^{\frac{1}{4}}.
\end{align*}

\no$\mathbf{Case\, 2.}$ If $0<\psi_{\tilde{x}}<(X^*-x)^{\frac{3}{4}}$, we only need to consider the case
\begin{align}\label{>14}
\partial_\psi w(\tilde{x},\psi_{\tilde{x}})>(X^*-\tilde{x})^{\frac{1}{4}}.
\end{align}

First of all, if $\partial_\psi w(\tilde{x},0)>\frac{1}{4}(X^*-\tilde{x})^{\frac{1}{4}},$ then we can replace $(\tilde{x},\psi_{\tilde{x}})$ by $(\tilde{x},0)$, and thus \eqref{km24.1} is satisfied by Theorem \ref{thm:lower}.

If $\partial_\psi w(\tilde{x},0)<\frac{1}{4}(X^*-\tilde{x})^{\frac{1}{4}},$ then we proceed as following. Let $(\tilde{x},y_{\tilde{x}})$ correspond to be the point  $(\tilde{x},\psi_{\tilde{x}})$ by Von Mises transformation. Set
\begin{align*}
  y_2=\inf\big\{y_1\in[0,y_{\tilde{x}}]:\partial_y u(\tilde{x},y)\geq\frac{1}{2}\partial_y u(\tilde{x},y_{\tilde{x}})\quad \text{in} \quad [y_1,y_{\tilde{x}}]\big\}.
\end{align*}
Thanks to $\partial_\psi w(\tilde{x},0)<\frac{1}{4}(X^*-\tilde{x})^{\frac{1}{4}},$ we have by \eqref{>14} and \eqref{PvM} that $y_2>0.$
Then we have
\begin{align}\label{yyum2}
\partial_y u(\tilde{x},y_{2})=\frac{1}{2}\partial_y u(\tilde{x},y_{\tilde{x}}).
\end{align}

 On the other hand, by the assumption $\partial_{y}^2u\leq M_0$ and Lemma \ref{lem:one side}, there exists a positive constant $M_1>1$ such that
\begin{align*}
-M_1\leq\partial_{y}^2u\leq M_1\quad[0,X^*)\times \mathbf{R}_+,
\end{align*}
which along with  \eqref{yyum2}  implies that
\begin{align}\label{yx-y2}
y_{\tilde{x}}-y_2\geq \frac{\partial_y u (\tilde{x},y_{\tilde{x}})}{2M_1}.
\end{align}
We have by \eqref{eq:VM}  that
\begin{align}\label{k3ym}
\begin{split}
 \psi(\tilde{x},y_{\tilde{x}})&=\int_0^{y_{\tilde{x}}}udy
 =\int_0^{y_{\tilde{x}}}\int_0^{y'}\partial_y u(\tilde{x},y'')dy''dy'
 \\&\geq
 \int_{y_2}^{y_{\tilde{x}}}\int_{y_2}^{y'}\partial_y u(\tilde{x},y'')dy''dy'
 \geq \frac{1}{2}\partial_y u (\tilde{x},y_{\tilde{x}})\frac{(y_{\tilde{x}}-y_2)^2}{2}\\
 &\geq \frac{1}{2}\partial_y u (\tilde{x},y_{\tilde{x}})\frac{1}{2}\Big(\frac{\partial_y u (\tilde{x},y_{\tilde{x}})}{2M_1}\Big)^2=c(\partial_y u  (\tilde{x},y_{\tilde{x}}))^3,
 \end{split}
\end{align}
where we used $u|_{y=0}=0$,  the definition of $y_2$ and  \eqref{yx-y2}. Hence, by \eqref{PvM}, we have
$$ \psi_{\tilde{x}}=\psi(\tilde{x},y_{\tilde{x}})\geq c(\partial_\psi w  (\tilde{x},\psi_{\tilde{x}}))^3,$$
which along with $\psi_{\tilde{x}}<(X^*-x)^{\frac{3}{4}}$ gives
$$
(X^*-x)^{\frac{3}{4}}>\psi_{\tilde{x}}\geq c(\partial_\psi w  (\tilde{x},\psi_{\tilde{x}}))^3.
$$

Summing up, we obtain the upper bound.
\end{proof}

The following proposition is inspired by Lemma 3.2 in
\cite{DM}, while we remove the structure assumption on initial data near $y=0$ there.

\begin{proposition}\label{prop:uyy-upper}
Let $u$ be a solution constructed in Proposition \ref{prop:olei} with $u_0$ satisfying $\partial_{y}^2u_0\leq 1$ and $X^*\leq x_0.$ Then it holds that
$\partial_{y}^2u\leq 1$ in $[0,X^*)\times \mathbf{R}_+.$
\end{proposition}

\begin{proof}
Let $f=\partial_x w =\sqrt{w}\partial^2_{\psi}w-2.$
By \eqref{PvM}, we have $2\partial^2_{y} u=\sqrt{w}\partial^2_{\psi}w.
$ Thus, we only need to show that
$$f\leq 0\quad\text{in}\quad[0,X^*)\times \mathbf{R}_+.$$

Assume that $\sup_{[0,X^*)\times \mathbf{R}_+} f>\epsilon_0$ for some $\epsilon_0>0$.  We define
$$
x_1=\inf\big\{x'\in[0,X^*)|\exists \psi_{x'}\in \mathbf{R}_+ \,\,\text{so that}\, \,f(x',\psi_{x'})\geq\frac{\epsilon_0}{2}\big\}.
$$
Due to $\partial_{y}^2u_0\leq 1,$ we have by \eqref{PvM} that
\begin{align}\label{pbx0}
   f|_{x=0}\leq 2-2\leq 0.
\end{align}
Hence, $0<x_1<X^*.$

In the following, we only consider $f$ in $[0,x_1]\times \mathbf{R}_+.$
It is easy to see that
\begin{align}\label{Mfeq}
 \partial_x f-\frac{f(f+2)}{2w}-\sqrt{w}\partial^2_{\psi}f=0 .
\end{align}
Now we consider the value of $f_+$ on ``parabolic boundary."
Thanks to $w|_{\psi=0}=0$ and Lemma \ref{lem:w-bound}, we have
\begin{align}\label{pbx1}
   f|_{\psi=0}=\partial_x w|_{\psi=0}=0,\quad \displaystyle\lim_{\psi\to+\infty} f(x,\psi)= -2\quad x\in[0,X^*).
\end{align}
Hence, by \eqref{pbx0} and \eqref{pbx1},
\begin{align}\label{f+pb}
f_+= 0\,\,\text{on}\,\, [0,X^*)\times\{\psi=0\}\cup\{x=0\}\times\mathbf{R}_+\cup[0,x_1]\times[K_{x_1},+\infty)
\end{align}
 for some large constant $K_{x_1}>0.$

Using \eqref{eq:VM} and \eqref{PvM}, a straight calculation yields
 $$\sqrt{w}\partial_{\psi}f=2u\frac{1}{u}\partial_{y}^3u=2\partial_{y}^3u.$$
As $|\partial_{y}^3u|\leq C_{x_1}$ and $f_+|_{\psi=0}= 0,$
we have
$$
\sqrt{w}f_+\partial_{\psi}f\rightarrow 0\quad\text{as}\quad\psi\rightarrow0.
$$
Then we get by integration by parts that
\begin{align}\label{pppp1}
 \int_{\mathbf{R}_+}\sqrt{w}f_+\partial_{\psi}^2f=
  -\frac{1}{2}\int_{\mathbf{R}_+}\partial_{\psi}(\sqrt{w})\partial_{\psi}(f_+)^2
  -\int_{\mathbf{R}_+}\sqrt{w}(\partial_{\psi}f_+)^2.
\end{align}

As $f|_{\psi=0}=2\partial_{y}^2u|_{y=0}-2=0$ and $|\partial_{y}^3u|\leq C_{x_1}$, we have $f\leq C_{x_1}y$ and hence,
 \begin{align}\label{f+y2}
 f_+^2\leq C_{x_1}y^2.
\end{align}
As $\partial_y u(x,0)>0$ on $[0,x_1]\times \mathbf{R}_+,$ there exist two positive constants $M_{x_1},m_{x_1}$ such that $$M_{x_1}>\partial_y u(x,0)>m_{x_1}\quad x\in[0,x_1].$$
Furthermore, there exists a positive constant $\delta_{x_1}$ so that
\begin{align}\label{quan1}
  2{ M_{x_1}}>\partial_y u(x,y)>\frac{m_{x_1}}{2}\quad x\in[0,x_1],\,y\in[0,\delta_{x_1}].
\end{align}
On the other hand, by \eqref{eq:VM} and $u|_{y=0}=0$, we have
\begin{align}\label{deltyx1}
 \psi(x,\delta_{x_1})&=\int_0^{\delta_{x_1}}udy
 \geq
 \int_{0}^{\delta_{x_1}}\int_{0}^{y'}\frac{m_{x_1}}{2}dy''dy'=c_{x_1}>0.
\end{align}
Then we conclude that for any $(x,\psi)\in[0,x_1]\times[0,c_{x_1}],$\begin{align*}
  4\frac{M_{x_1}}{m_{x_1}}\frac{1}{y}\geq\frac{\partial_\psi w }{2\sqrt{w}}=\frac{2\partial_y u(x,y)}{2u}\geq\frac{m_{x_1}}{4M_{x_1}}\frac{1}{y},
\end{align*}
which along with \eqref{f+y2} gives
\begin{align*}
\frac{\partial_\psi w }{2\sqrt{w}}f_+^2\rightarrow0\quad\text{as}\,\,\psi\rightarrow0.
\end{align*}
Then we get by integration by parts that
\begin{align}
-\frac{1}{2}\int_{\mathbf{R}_+}\partial_{\psi}(\sqrt{w})\partial_{\psi}(f_+)^2
=\frac{1}{2}\int_{\mathbf{R}_+}\partial_{\psi}^2(\sqrt{w})(f_+)^2.
\end{align}
Note that
\beno
&&\partial_{\psi}^2(\sqrt{w})=\frac{(f+2)}{2w}-\frac{1}{4}\frac{(\partial_\psi w )^2}{w^{\frac{3}{2 }}},\\
&&\frac{f(f+2)}{2w}f_+\leq \big(\frac{\epsilon_0}{2}+2\big)\frac{(f_+)^2}{2w}\quad\text{on}\quad[0,x_1]\times \mathbf{R}_+.
\eeno
The we infer that
\begin{align}\label{finaint}\begin{split}
 \frac{1}{2}\frac{d}{dx}\int_{\mathbf{R}_+}(f_+)^2+\frac{1}{8}\int_{\mathbf{R}_+}\frac{(\partial_\psi w )^2}{w^{\frac{3}{2 }}}(f_+)^2+\int_{\mathbf{R}_+}\sqrt{w}(\partial_\psi f_+)^2\leq C_{x_1}\int_{\mathbf{R}_+}\frac{(f_+)^2}{w}.\end{split}
\end{align}

By \eqref{quan1}  and \eqref{PvM}, we have
\begin{align}\label{quan2}
 4M_{x_1}>\partial_\psi w (x,\psi)>m_{x_1}\quad x\in[0,x_1],\,\psi\in[0,c_{x_1}].
\end{align}
By $w|_{\psi=0}=0$ and \eqref{deltyx1}, for fixed large $K$, there exists a positive constant $\psi^{x_1}_0<c_{x_1}$ so that for $x\in[0,x_1],\psi\in[0,\psi^{x_1}_0]$,
\begin{align*}
\frac{(\partial_\psi w (x,\psi))^2}{w^{\frac{3}{2 }}}\geq \frac{(m_{x_1})^2}{(4M_{x_1})^{\frac{3}{2 }}\psi^{\frac{3}{2 }}}\geq\frac{K}{m_{x_1}\psi}\geq \frac{K}{w}.
\end{align*}
On the other hand, since $w$ is a non-decreasing function, we have
\begin{align*}
   w\geq m_{x_1}\psi^{x_1}_0=\tilde{c}_{x_1}>0\quad \text{on}\quad [\psi^{x_1}_0,+\infty).
\end{align*}
Then we infer that
\begin{align*}
&C_{x_1}\int_{0}^{\psi^{x_1}_0}\frac{(f_+)^2}{w}\leq \frac{1}{8}\int_{\mathbf{R}_+}\frac{(\partial_\psi w )^2}{w^{\frac{3}{2 }}}(f_+)^2,\\
&C_{x_1}\int_{\psi^{x_1}_0}^{+\infty}\frac{(f_+)^2}{w}\leq C_{x_1}\int_{\mathbf{R}_+}(f_+)^2,
\end{align*}
which along with \eqref{finaint} give
\begin{align}
    \frac{1}{2}\frac{d}{dx}\int_{\mathbf{R}_+}(f_+)^2\leq C_{x_1}\int_{\mathbf{R}_+}(f_+)^2.
\end{align}
Since $f_+=0$ on $\{x=0\}\times\mathbf{R}_+$, by Gronwall's inequality, we have $f_+=0$ in $[0,x_1]\times\mathbf{R}_+,$ which is a contradiction to the definition of $x_1$, and thus the proof is completed.
\end{proof}

\begin{remark}
In the proof of lemma, we use $|\partial_{y}^3u|\leq C_{x_1}$ in order to show that $\sqrt{w}f_+\partial_{\psi}f$ and $\frac{\partial_\psi w }{2\sqrt{w}}f_+^2$ vanish on $\psi=0$. However, we don't have this information in Proposition \ref{prop:olei}. We can make it rigorous by the following argument.

By Theorem 2.1.14 in \cite{Olei}, we have $|\partial_x w|\leq M_{x_1}\psi^{1-\beta},\,0\leq x\leq x_1, \,0\leq \psi\leq \psi_1,$ for some $0<\beta<\frac{1}{2}.$ For convenience, we denote $\beta=\frac{1}{2}-\alpha$ where $0<\alpha<\frac{1}{2}.$ Then
\begin{align}\label{xd}
|\partial_x w|\leq M_{x_1}\psi^{\frac{1}{2}+\alpha},\,0\leq x\leq x_1,\, 0\leq \psi\leq \psi_1.
\end{align}
As $f=\partial_x w$, then we get
 \begin{align}\label{qcxf}
 f_+\leq M_{x_1}\psi^{\frac{1}{2}+\alpha},\,\quad 0\leq x\leq x_1,\, 0\leq \psi\leq \psi_1.
\end{align}
This implies that for any $x\in[0,x_1],$
$$
\frac{\partial_\psi w }{2\sqrt{w}}f_+^2\leq M_{x_1}\psi^{\frac{1}{2}+2\alpha}\rightarrow 0\quad\text{as}\,\,\psi\rightarrow 0.
$$ 
Take $\{\psi_n\}$ such that $\psi_n\rightarrow0,\, n\rightarrow+\infty$ and $\psi_n\in(0,\psi_1].$ Then by mean value theorem, for any $x\in[0,x_1],$
there exists a family of points $\{(x,\xi_n^x)\}$ with $\xi_n^x\in(0,\psi_n)$ so that
\begin{align}\label{xf}\begin{split}
|(\partial_\psi(\sqrt{w}\partial_xw))(x,\xi_n^x)|&=\Big|\frac{\sqrt{w}\partial_xw(x,\psi_n)-\sqrt{w}\partial_xw(x,0)}{\psi_n}\Big|
\\&=\Big|\frac{\sqrt{w}\partial_xw(x,\psi_n)}{\psi_n}\Big|\leq \psi_n^{\alpha}\rightarrow 0, \quad n\rightarrow+\infty.
\end{split}
\end{align} 
Note that 
$$\partial_\psi(\sqrt{w}\partial_xw)=\sqrt{w}\partial^2_{x\psi}w+\frac{\partial_\psi w}{2\sqrt{w}}\partial_xw,$$ 
and by \eqref{xd}, for any $x\in[0,x_1],$
$\Big|\frac{\partial_\psi w}{2\sqrt{w}}\partial_xw\Big|\leq M_{x_1}\psi^{\frac{1}{2}+\alpha-\frac{1}{2}}\rightarrow0$ as $\psi\rightarrow0$. Then we have by \eqref{xf} that $\sqrt{w}\partial^2_{\psi x}w(x,\xi_n^x)\rightarrow0$ as $n\rightarrow+\infty.$ Therefore, $\sqrt{w}\partial_{\psi}f(x,\xi_n^x)\rightarrow0$ as $n\rightarrow+\infty$. This shows that 
 for any $x\in[0,x_1],$ 
 $$\sqrt{w}f_+\partial_{\psi}f(x,\xi_n^x)\rightarrow0\quad \text{as}\,\,\,n\rightarrow+\infty.$$

\end{remark}
\section{General adverse pressure gradient}\label{nonconstant}

In this section, we present some key changes for general adverse pressure $p$, which holds  that for some constants $C,c>0,$
$$0<c<\partial_xp <C,\quad |\partial^2_{x}p|\leq C.$$
Now $(U,p)$ satisfies the Bernoulli equation
\beno
U(x)^2=Const-2p(x).
\eeno
As $0<c<\partial_xp <C,$ there exists a point $x_0$ such that $U(x_0)=0$ and $U(x)>0,x<x_0$. We denote the first vanishing point of $U$ by $x_0.$ Then we have $Const-2p(x_0)=U(x_0)=0,$ which implies
\begin{align*}
\begin{split}U(x)^2=2p(x_0)-2p(x)=2(p(x_0)-p(x))=2\int_x^{x_0}\frac{dp(x)}{dx}dx.
\end{split}\end{align*}
Therefore,
\begin{align}\label{keyU}
  2c(x_0-x)\leq U(x)^2\leq 2C(x_0-x).
\end{align}

Using Von Mises transformation \eqref{eq:VM}, we have
\begin{align}\label{MSPnonc}
    \begin{split}
     &w_x-\sqrt{w}\partial^2_{\psi\psi}w=-2\frac{dp(x)}{dx} \quad \text{in}\quad (0,X_*)\times(0,+\infty), \\&w_0(\psi)=u_0^2(y(\psi)),\quad\,w(x,0)=0,
       \quad\displaystyle\lim_{\psi\to+\infty} w(x,\psi)= U(x)^2.
    \end{split}
 \end{align}

Now we explain some key changes in the proof.

\begin{itemize}

\item[1.] For the proof of the key Lemma \ref{lem:one side}, the main difference is that, at $(x,\psi)=z_{min}\in D_{\bar{x}}$ where $f=\pa_xw$ takes its negative minimum, we have
$$c\frac{f}{w}+\frac{f^2}{2w}<\frac{dp(x)}{dx}\frac{f}{w}+\frac{f^2}{2w} \leq 2\frac{d^2p(x)}{dx^2} \leq 2C,$$
which also implies
$$-C_1\leq f(z_{min})<0.$$
For the other parts in section 3, we just use the fact $\pa_xp\ge c>0$.

\item[2.] The arguments in section 4 just rely on  Lemma \ref{lem:one side}.

\item[3.] The proof of Lemma \ref{lem:lower}(thus Theorem \ref{thm:behavior})  just uses the fact that $\pa_xp\ge c>0$.

\item[4.]  For Proposition \ref{prop:uyy-upper}, the condition $\partial_{y}^2u_0\leq 1$ is replaced by
\begin{align}
\partial_{y}^2u_0\leq \frac{dp}{d x}(0).\label{pxf+leq0}
\end{align}
Moreover, we need to assume that
\begin{align}
\frac{d^2p}{d x^2}\ge 0.\label{extravp}
\end{align}
The main reason is that if $\frac{d^2p}{dx^2}\neq 0,$ then \eqref{Mfeq} becomes
\begin{align*}
 \partial_x f-\frac{f(f+2\frac{dp}{d x})}{2w}-\sqrt{w}\partial^2_{\psi}f+2\frac{d^2p}{d x^2}=0,
\end{align*}
where $f=\pa_xw.$ The condition \eqref{pxf+leq0} ensures that $f_+\leq 0$ on $\{x=0\}\times\mathbf{R}_+$, and the sign of $\frac{d^2p}{d x^2}$ in \eqref{extravp} ensures that the inequality \eqref{finaint} still holds so that we can draw the conclusion
$$\partial_{y}^2u\leq C.
$$
\end{itemize}

\section*{Acknowledgments}

 Z. Zhang is partially supported by NSF of China under Grant 11425103.

 \end{document}